\documentclass[11pt]{article}
\usepackage{amscd,amssymb,stmaryrd}
\usepackage{amsmath}
\usepackage{graphicx}
\usepackage{subcaption}
\usepackage[utf8]{inputenc}
\usepackage[export]{adjustbox}
\usepackage{wrapfig}
\usepackage{amstext}
\usepackage{pstricks,pst-node,pst-plot,pst-coil}
\usepackage{amsthm}
\usepackage{amsmath}
\usepackage{color}
\usepackage{mathtools}
\usepackage{hyperref}
\usepackage[english]{babel}
\usepackage{booktabs}
\usepackage{mathrsfs}
\usepackage{comment}
\usepackage{float}
\usepackage[linesnumbered,ruled,vlined, ruled,vlined]{algorithm2e}

\theoremstyle{definition}
\newtheorem{theorem}{Theorem}[section]

\newtheorem{lemma}[theorem]{Lemma}
\newtheorem{example}[theorem]{Example}

\theoremstyle{remark}

\usepackage[english]{babel}

\usepackage{natbib}

\usepackage{amssymb}

\newcommand{\T}{\mathrm{T}}

\DeclareMathOperator*\trace{trace}
\DeclareMathOperator*\diag{diag}

\newcommand{\cD}{\mathcal{D}}

\newcommand{\bbD}{\mathbb{D}}

\newcommand{\cE}{\mathcal{E}}

\newcommand{\bbE}{\mathbb{E}}

\newcommand{\cF}{\mathcal{F}}

\newcommand{\cI}{\mathcal{I}}

\newcommand{\cN}{\mathcal{N}}

\newcommand{\bbR}{\mathbb{R}}

\def\FT{f-reSGLD}

\title{Fast Replica Exchange Stochastic Gradient Langevin Dynamics}

\date{}

\author{Guanxun Li\footnote{Department of Statistics, Texas A\&M University, College Station, TX 77843, USA. (Email: guanxun@stat.tamu.edu)}, ~Guang Lin \footnote{Department of Mathematics, Purdue University, West Lafayette, IN 47907, USA. (Email: guanglin@purdue.edu)}, 
~Zecheng Zhang \footnote{Department of Mathematics, Carnegie Mellon University, Pittsburgh, PA 15213, USA. (Email: zecheng.zhang.math@gmail.com)}, ~Quan Zhou \footnote{Department of Statistics, Texas A\&M University, College Station, TX 77843, USA. (Email: quan@stat.tamu.edu) } }

\begin{document}

\maketitle

\begin{abstract}
Application of the replica exchange (i.e., parallel tempering) technique to Langevin Monte Carlo algorithms, especially stochastic gradient Langevin dynamics (SGLD), has scored great success in non-convex learning problems, but one potential limitation is the computational cost caused by running multiple chains. 
Upon observing that a large variance of the gradient estimator in SGLD essentially increases the temperature of the stationary distribution, we propose expediting tempering schemes for SGLD by directly estimating the bias caused by the stochastic gradient estimator. 
This simple idea enables us to simulate high-temperature chains at a negligible computational cost (compared to that of the low-temperature chain) while preserving the convergence to the target distribution.   
Our method is fundamentally different from the recently proposed m-reSGLD (multi-variance replica exchange SGLD) method in that the latter suffers from the low accuracy of the gradient estimator (e.g. the chain can fail to converge to the target) while our method benefits from it. 
Further, we derive a swapping rate that can be easily evaluated, providing another significant improvement over m-reSGLD. 
To theoretically demonstrate the advantage of our method, we develop convergence bounds in Wasserstein distances.  
Numerical examples for Gaussian mixture and inverse PDE models are also provided, which show that our method can converge quicker than the vanilla multi-variance replica exchange method. 
\end{abstract}

\section{Introduction}
Given a probability distribution $\pi(\theta) \propto \exp(- U(\theta))$, where the energy function $U$ is assumed known, one can use Markov chain Monte Carlo (MCMC) methods to generate samples from $\pi$. When $\pi$ is log-concave (i.e., $U$ is convex), it is known that most MCMC algorithms used in practice (e.g. Metropolis-Hastings and Hamiltonian Monte Carlo algorithms) are rapidly mixing~\citep{dwivedi2018log, mangoubi2021mixing}. 
However, in reality, we are often faced with much more challenging problems where $\pi$ tends to be severely multimodal. 
One of the most widely used techniques for overcoming multimodality is parallel tempering, which is also known as replica exchange Monte Carlo~\citep{machta2009strengths}. 
The main idea is that, instead of targeting $\pi$, one can devise MCMC algorithms targeting $\pi_\tau (\theta) \propto \exp(- U(\theta) / \tau)$, where the parameter $\tau > 0$ is often called the temperature. 
When $\tau$ is sufficiently large, $\pi_\tau$ has a flat shape such that the chain can move between local modes without much difficulty. 
Replica exchange, in its simplest form, runs two MCMC algorithms in parallel with one targeting $\pi_{\tau_1}$ for some small $\tau_1$ (i.e., the ``low-temperature'' chain) and the other targeting $\pi_{\tau_2}$ for some large $\tau_2$ (i.e., the ``high-temperature'' chain). The low-temperature chain can quickly find the nearby local mode, while the high-temperature one keeps exploring the whole state space. 
By allowing the two chains to swap their states, the low-temperature chain is then able to efficiently jump between local modes. To recover the distribution $\pi$, one can simply set $\tau_1 = 1$ and collect the samples of the low-temperature chain, or use the importance tempering method~\citep{gramacy2010importance}. 
Compared with single-chain MCMC methods, one potential limitation of this scheme is that the computational cost doubles in every iteration.  

\subsection{Background and Motivations}\label{sec:intro.reSGLD}
We consider in this paper using Langevin Monte Carlo methods for simulating each  chain. 
Under some regularity assumptions on $U$, we can construct a Langevin diffusion (LD)  with stationary distribution $\pi_\tau$~\citep{roberts2002langevin, nguyen2019non}; see Eq.~\eqref{eq:LD-SDE}.  
For huge data sets, exactly evaluating $U$ or the gradient of $U$ (which is needed to simulate LD) can be quite time-consuming, and a popular approach used in both the sampling and optimization literature is to estimate $U$ and $\nabla U$ using a random batch of samples~\citep{welling2011bayesian, dalalyan2017further, simsekli2020fractional}; this method is known as stochastic gradient Langevin dynamics (SGLD). 
We note that other methods can also be employed for the gradient estimation; for example, for inverse PDE problems~\citep{efendiev2006preconditioning, stuart2010inverse, chung2020multi}, $U$ is calculated via a forward PDE solver, and to save computational time, one may use a fast solver that only provides an approximation of $U$. 
For non-convex optimization and multimodal sampling problems, a vanilla Langevin Monte Carlo algorithm can get trapped at local modes, and the replica exchange LD (reLD) and replica exchange SGLD (reSGLD) algorithms were proposed to combine replica exchange with Langevin dynamics~\citep{chen2020accelerating, deng2020non}. Both methods run two chains in parallel, and it has been shown theoretically that reLD and reSGLD require fewer iterations to converge than single-chain methods~\citep{zhang2017hitting, raginsky2017non, deng2020non, chen2020accelerating}.

However, since reSGLD employs two chains, its computational cost per iteration doubles and its total computational cost can possibly be higher than that of the single-chain SGLD (one can of course generalize reSGLD by running more than two chains in parallel, in which case the total computational cost may be much higher). 
To tackle this issue, multi-variance reSGLD~\citep{lin2022multi, lin2021accelerated, na2022replica} (m-reSGLD) was proposed aiming to lower the cost of the high-temperature chain. 
Unlike reSGLD,  m-reSGLD uses different gradient estimation schemes for the two chains, and the estimator for the high-temperature chain is computationally more efficient (and also has a larger mean squared error) than that for the low-temperature chain. 
Since the high-temperature chain is used for exploration and the low-temperature one for exploitation, one expects that the high-temperature chain can probably tolerate a larger degree of error in gradient estimation. 
This intuition is supported by the empirical success of m-reSGLD in many challenging tasks~ \citep{lin2022multi, lin2021accelerated}. 
Nevertheless, the advantage of m-reSGLD over reSGLD is not theoretically justified. More importantly, it is unclear how the gradient estimation error in the high-temperature chain affects the overall performance of the algorithm, and how large this error is allowed to be so that the high-temperature chain is still exploring the whole space according to the tempered target distribution. 

\subsection{Main Contributions of This Work}\label{sec:our.contributions} 
We propose a simple but highly effective method, with theoretical guarantees,  for simulating high-temperature SGLD chains at a negligible computational cost. 
The key idea is to learn the ``effective temperature'' of an SGLD chain by estimating the variance of the gradient estimator.  
Let $v(\theta)$ denote the covariance matrix of the gradient estimator $\widehat{\nabla U}(\theta)$, which may depend on the state $\theta$.  Although a larger $v(\theta)$ is undesirable from an estimation perspective, for high-temperature chains in a replica exchange scheme, it brings the randomness that is needed to encourage exploration. Indeed, when we numerically simulate a discrete-time approximation of the Langevin diffusion, at each time step we need to inject random noise. Merging the noise of $\widehat{\nabla U}(\theta)$  with the injected noise, we see that the high variability of $\widehat{\nabla U}(\theta)$ essentially increases the temperature of the chain. From a different angle, this also implies that to achieve a fixed temperature, instead of injecting a large random noise, we can also increase $v(\theta)$ by using a ``rougher'' gradient estimator. 
The details of the derivation are given in Section~\ref{sec_pp}. 
We note that, for single-chain SGLD methods, the bias correction of the gradient estimator has been well-studied~\citep{teh2016consistency, vollmer2016exploration}. 
The novelty of our method is that we deliberately use this bias (and even increase it) to run a high-temperature SGLD chain at a much smaller cost. 
We call our method \FT{} (fast replica exchange SGLD).

For most target distributions encountered in practice,  $v(\theta)$ is unknown, but it is usually not too difficult to construct a ``good'' estimator $\widehat{v}(\theta)$ (see Eq.~\eqref{eq:var-est}), in which case our theoretical analysis shows that the computational gain of our method can be very significant.  
Further, we derive rates of convergence of \FT{} to the target distribution in 2-Wasserstein distance. Compared to~\citet[Theorem 1]{deng2020non}, the error due to the gradient estimation, which is the dominant term, is reduced. 
Though replica exchange is mostly used to solve non-convex problems with multiple local optima, it may also provide an efficient solution to convex problems, especially when a good initialization is not available. We derive a convergence result when the energy function $U(\theta)$ is strongly convex. Compared to the non-convex version of the convergence estimation, the error term due to the estimation of the gradient will vanish as the learning rate goes to zero.

We summarize our main contributions as follows.
 
\begin{enumerate}
\item We propose simulating high-temperature SGLD chains by treating the variance of the gradient estimator as the main source of randomness, which can significantly reduce the computational cost. 
\item We propose a swapping rate (between two chains) for the \FT{} algorithm, which is computationally more efficient than the rule used in m-reSGLD~\citep{lin2022multi}.
\item We prove convergence bounds in 2-Wasserstein distance for both convex and non-convex target distributions, which illustrate the advantage of the \FT{} algorithm. 
\end{enumerate}

The rest of the paper is organized as follows. In Section~\ref{sec_prelim}, we review SGLD,  reSGLD and m-reSGLD algorithms.  
We formally introduce our method in Section~\ref{sec_pp} and present the theoretic convergence results in Section~\ref{sec_convergence}. Numeric experiments are given in Section~\ref{sec_exp}. Finally, we conclude our work with a discussion in Section~\ref{sec_con}. All proofs are relegated to the Appendix.

\section{Preliminaries}\label{sec_prelim}
\subsection{Stochastic gradient Langevin dynamics}\label{sec:prelim.sgld}
Suppose we are interested in sampling from a distribution 
\[\pi_{\tau}(\theta ) \propto \exp (-U(\theta )/ \tau),\]
where $\theta \in\bbR^p$, $\tau > 0$ is the temperature and the energy function $U$ is known.   
One method for generating samples from $\pi_{\tau}(\theta )$ is to simulate a stochastic process whose stationary distribution is $\pi_{\tau}$. 
Under certain regularity conditions (see, e.g.~\citet{bhattacharya1978criteria,roberts1996exponential}),  the following stochastic differential equation (SDE), known as Langevin diffusion (LD), has stationary distribution $\pi_\tau$: 
\begin{equation}\label{eq:LD-SDE}
d\theta_t  = -\nabla U(\theta_t )\,dt + \sqrt{2\tau}d W_t,
\end{equation}
where $W_t$ is the $p$-dimensional standard Brownian motion.  To perform sampling in practice, we discrete \eqref{eq:LD-SDE} by 
\begin{equation}\label{eq:LD_discret}
\theta_{k + 1, \eta}  = \theta_{k, \eta}  - \eta_k \nabla U(\theta_{k, \eta} ) + \sqrt{2 \eta_k \tau}\xi_k,
\end{equation}
where $\xi_k\sim\cN(0, I_p)$ and $\eta_k$ is the step size used in the $k$-th step. We will refer to~\eqref{eq:LD_discret} as the ``exact update''. In many problems, $U(\theta)$ can be further written as  $U(\theta ) = \sum_{i = 1}^N U_i(\theta )$ for some functions $U_1, U_2, \dots, U_N$, and 
\begin{equation}\label{eq:full-deriv}
\nabla U(\theta ) = \sum_{i = 1}^N \nabla U_i(\theta ).
\end{equation} 

\begin{example}
Consider Bayesian inference with $N$ i.i.d. observations. Let $\cD = \{X_i\}_{i=1}^N$ be the data we observe and $f( \cdot | \theta )$ be the probability density function of an observation given parameter $\theta$. Let $\Pi(\theta )$ be the prior distribution we put on $\theta$. The posterior distribution, denoted by $\Pi(\theta | \cD)$, can be computed by 
\[\Pi(\theta | \cD) \propto \Pi(\theta)\prod_{i = 1}^N f(X_i | \theta).\]
Letting $\pi(\theta) = \Pi(\theta | \cD)$, we can express the energy function by 
\[U(\theta) = -\log \Pi(\theta) - \sum_{i = 1}^N \log f(X_i | \theta) = \sum_{i = 1}^N U_i(\theta),\]
where $U_i(\theta) = -\frac{1}{N}\log \Pi(\theta) - \log f(X_i | \theta)$.
\end{example}

When the sample size $N$ is large, implementing \eqref{eq:LD_discret} can be time-consuming due to the evaluation of $\nabla U$.  Stochastic gradient Langevin diffusion (SGLD) then can be used to speed up the LD simulation. Instead of exactly calculating $\nabla U$ by~\eqref{eq:full-deriv}, we sample $n$ observations from $\cD$, and estimate $\nabla U(\theta)$ by
\begin{equation}\label{eq:sgld.estimator}
\widehat{\nabla U}(\theta) \coloneqq  \frac{N}{n}\sum_{i\in\cI_n}\nabla U_i(\theta),
\end{equation}
where $\cI_n$ is a random sample of size $n$ drawn from $\{1, \dots, N\}$ without replacement. The update of SGLD targeting $\pi_\tau$ is then given by 
\[\widehat \theta_{k+1, \eta} = \widehat \theta_{k, \eta} - \eta_k \widehat{\nabla U}(\widehat \theta_{k, \eta}) + \sqrt{2\eta_k \tau}\xi_k.\]
In SGLD, the step size $\eta_k$ is chosen such that $\eta_k \rightarrow 0$ as $k \rightarrow \infty$. We note that if the goal is to generate samples from $\pi_\tau$, one can also use the SGLD dynamics by choosing a constant step size and using the Metropolis-Hastings rule to correct for the discretization bias, which is known as the Metropolis-adjusted Langevin algorithm~\citep{besag1994comments,roberts1998optimal}; in particular, as long as one has an unbiased estimator of $e^{-U / \tau}$, the pseudo-marginal MCMC technique can be used to evaluate the acceptance ratio~\citep{andrieu2009pseudo}. 

\subsection{Replica Exchange SGLD}
\label{sec_resgld}
Replica exchange Langevin diffusion (reLD) is an algorithm that aims to accelerate the convergence of the SDE when the target is non-convex or multimodal~\citep{chen2020accelerating}. Letting $\tau_1, \tau_2$ denote two temperatures with $\tau_2 > \tau_1$, define two parallel LDs by 
\begin{align}\label{eq:reLD}
\begin{split}
d\theta^{(1)}_t &= -\nabla U(\theta^{(1)}_t)\,dt + \sqrt{2\tau_1}d W_t^{(1)},\\
d\theta^{(2)}_t &= -\nabla U(\theta^{(2)}_t)\,dt + \sqrt{2\tau_2}d W_t^{(2)},
\end{split}
\end{align}
where $W_t^{(1)}$ and $W_t^{(2)}$ are two independent Brownian motions. The 
reLD algorithm further allows two LDs to swap their states, i.e., moving from 
\[(\theta^{(1)}_t = y_1, \theta^{(2)}_t = y_2) \text{ to }(\theta^{(1)}_{t + dt} = y_2, \theta^{(2)}_{t + dt} = y_1)\]
with probability $a\min\{1, S(\theta^{(1)}_t, \theta^{(2)}_t)\}\,dt$, where $a > 0$ is a constant, 
\begin{equation}\label{eq:swap-reLD}
S(\theta^{(1)}, \theta^{(2)}) \coloneqq \exp\left\{\tau_{\delta}\left(U(\theta^{(1)}) - U(\theta^{(2)})\right)\right\},    
\end{equation}
and $\tau_{\delta} = 1 / \tau_1 - 1 / \tau_2$. It is well known that the stationary distribution of reLD is \citep{chen2020accelerating}
\begin{equation}\label{eq:stationary-dis}
\pi_{\rm{re}}(\theta_t^{(1)}, \theta^{(2)}_t) \propto \exp\left\{-\frac{U(\theta_t^{(1)})}{\tau_1} - \frac{U(\theta_t^{(2)})}{\tau_2}\right\}.  
\end{equation}

\citet{deng2020non} first proposed replica exchange stochastic gradient Langevin diffusion (reSGLD). Suppose we have an energy function estimator $ \widehat U(\theta)$ and a gradient estimator $\widehat{\nabla U}(\theta)$.  
We simulate a discrete-time approximation of two tempered LDs by 
\begin{align}\label{eq:reSGLD}
\begin{split}
 \widehat \theta_{k + 1, \eta}^{(1)} &=  \widehat \theta_{k, \eta}^{(1)} - \eta_k \widehat{\nabla U}( \widehat \theta_{k, \eta}^{(1)}) + \sqrt{2\eta_k\tau_1}\xi_k^{(1)}\\
 \widehat \theta_{k + 1, \eta}^{(2)} &=  \widehat \theta_{k, \eta}^{(2)} - \eta_k  \widehat{ \nabla U}( \widehat \theta_{k, \eta}^{(2)}) + \sqrt{2\eta_k\tau_2}\xi_k^{(2)},
\end{split}
\end{align}
where $\xi_k^{(1)}$ and $\xi_k^{(2)}$ independently follow the standard normal distribution. 
Assuming $\widehat{U}(\theta) \sim \cN(U(\theta), \sigma^2)$, \cite{deng2020non} proposed to use the swapping rate 
$a\eta_k\min\{1,  \widehat S( \widehat \theta^{(1)}_{k, \eta},  \widehat \theta^{(2)}_{k, \eta})\}$, where
\begin{equation}\label{eq:reSGLD-swap}
\widehat S(\theta^{(1)}, \theta^{(2)})   \coloneqq  e^{   \tau_{\delta}\left( \widehat U(\theta^{(1)}) -  \widehat U(\theta^{(2)}) - \tau_{\delta}\sigma^2\right) }.
\end{equation} 
Since $\bbE[ e^{ b \widehat{U} (\theta) } ] = e^{b U(\theta) + b^2 \sigma^2 / 2}$ for any $b \in \bbR$, we have $\bbE[ \widehat{S} (\theta^{(1)}, \theta^{(2)})] = S (\theta^{(1)}, \theta^{(2)})$. This swapping rate is not exactly ``unbiased'', since by Jensen's inequality, 
\begin{align*}
    \bbE[ \min\{ 1 , \,  \widehat{S} (\theta^{(1)}, \theta^{(2)}) \}] \leq \min\{1,  S (\theta^{(1)}, \theta^{(2)})\}, 
\end{align*}
and the strict inequality holds when $\sigma > 0$. Empirically, it was found in~\citet{deng2020non} that this rule works well for most problems. 

\citet{lin2022multi} generalized reSGLD by using different estimators for the two chains. More specifically, suppose we have 
\[\widehat U_1(\theta) \sim \cN(U(\theta), \sigma_1^2) \text{ and } \widehat U_2(\theta) \sim \cN(U(\theta), \sigma_2^2),\]
where $\widehat U_1$ is the energy function estimator used in the low-temperature chain and $\widehat U_2(\theta)$ is used in the high-temperature chain. The m-reSGLD algorithm updates the two discrete-time processes $\widehat{\theta}^{(1)}, \widehat{\theta}^{(2)}$ by 
\begin{align}\label{eq:multi-reSGLD}
\begin{split}
\widehat \theta_{k + 1, \eta}^{(1)} &= \widehat \theta_{k, \eta}^{(1)} - \eta_k \widehat{\nabla U}_1(\widehat \theta_{k, \eta}^{(1)}) + \sqrt{2\eta_k\tau_1}\xi_k^{(1)}\\
\widehat \theta_{k + 1, \eta}^{(2)} &= \widehat \theta_{k, \eta}^{(2)} - \eta_k \widehat{\nabla U}_2(\widehat \theta_{k, \eta}^{(2)}) + \sqrt{2\eta_k\tau_2}\xi_k^{(2)}
\end{split}
\end{align}
with swapping rate $a\eta_k\min\{1, \widehat S(\widehat \theta^{(1)}_{k, \eta}, \widehat \theta^{(2)}_{k, \eta})\}$. Letting $a_1, a_2$ be two non-negative constants such that $a_1 + a_2 = 1$, we can calculate the function $\widehat S(\theta^{(1)}, \theta^{(2)})$ by 
\begin{align}\label{eq:multi-reSGLD-swap}
\begin{split}
 \widehat S(\theta^{(1)}, \theta^{(2)}) = & \exp \Bigl\{ \tau_{\delta}\Bigl[ a_1\left(\widehat U_1(\theta^{(1)}) - \widehat U_1(\theta^{(2)})\right) + a_2\left(\widehat U_2(\theta^{(1)}) - \widehat U_2(\theta^{(2)})\right) \\
& - (a_1^2 \sigma_1^2 + a_2^2 \sigma_2^2) \tau_{\delta} \Bigl] \Bigr\}. 
\end{split}  
\end{align} 
By a straightforward calculation, one can verify that 
\[\bbE[\widehat S(\theta^{(1)}, \theta^{(2)})] = S(\theta^{(1)}, \theta^{(2)}).\] 

\section{Fast Tempering for SGLD via Bias Correction}
\label{sec_pp}
We now propose our method, \FT{}, which improves on reSGLD in terms of both accuracy and efficiency.  
To simplify the discussion, as in~\citep{deng2020non, lin2022multi}, we assume our estimator of gradient follows a normal distribution, i.e., $\widehat{\nabla U}(\theta) \sim \cN(\nabla U(\theta), s(\theta)s(\theta)^\top)$ for some positive definite $s(\theta) \in \bbR^{p \times p}$.   
Note that we allow the covariance matrix of the estimator to depend on $\theta$. 

First, consider a single SGLD chain targeting $\pi_\tau \propto e^{-U / \tau}$. We propose to simulate the discrete-time process $\widehat{\theta}$ with dynamics given by 
\begin{equation}\label{eq:adj-LD}
\widetilde \theta_{k+1, \eta} = \widetilde \theta_{k, \eta} - \eta_k \widehat{\nabla U}(\widetilde \theta_{k, \eta}) + \sqrt{2} c_k (\widetilde \theta_{k, \eta})\xi_k,
\end{equation}
where $\xi_k\sim\cN(0, I_p)$, and the matrix $c_k(\theta)$ is assumed to be positive definite such that 
\begin{equation}\label{eq:def-c}
c_k(\theta) c_k(\theta)^\top \coloneqq  \tau \eta_k I_p - \frac{1}{2} \eta^2_k s(\theta)s(\theta)^\top. 
\end{equation}
This can always be satisfied by letting $\eta_k$ be sufficiently small. 
To see the reasoning behind \eqref{eq:def-c}, notice that we can rewrite \eqref{eq:adj-LD} as
\begin{align*}
\widetilde \theta_{k + 1, \eta} =\,& \widetilde \theta_{k, \eta}  - \eta_k \nabla U(\widetilde \theta_{k, \eta})  + \eta_k (\nabla U(\widetilde \theta_{k, \eta}) - \widehat{\nabla U}(\widetilde \theta_{k, \eta})) + \sqrt{2} c_k(\widetilde \theta_{k, \eta})\xi_k\\    
=\, & \widetilde \theta_{k, \eta} - \eta_k \nabla U(\widetilde \theta_{k, \eta}) - \eta_k s(\widetilde \theta_{k, \eta}) \zeta_k  + \sqrt{2} c_k(\widetilde \theta_{k, \eta}) \xi_k,
\end{align*}
where $\zeta_k \sim \cN(0, I_p)$ and independent of $\xi_k$. By~\eqref{eq:def-c}, we see that~\eqref{eq:adj-LD} is the same as the exact upate~\eqref{eq:LD_discret}. 
This is known as the bias correction for SGLD methods~\citep{teh2016consistency, vollmer2016exploration}. 

In practice, usually, we do not know the true covariance matrix $s(\theta)s(\theta)^\top$, but we can estimate it. Suppose we have an estimator $\widehat{s}(\theta)$ such that 
\[s(\theta)s(\theta)^\top = \widehat{s}(\theta) \widehat{s}(\theta)^\top + \varphi(\theta) \varphi(\theta)^\top,\]
where $\varphi(\theta)$ satisfies 
\begin{equation}\label{eq:var-est}
\trace ( \varphi(\theta)\varphi(\theta)^\top ) = o( \trace ( s(\theta) s(\theta)^\top)  )     
\end{equation}
for any $\theta$. In~\eqref{eq:var-est}. Note that 
\[\trace ( \varphi(\theta)\varphi(\theta)^\top ) = \bbE[\| \varphi (\theta) \zeta \|_2^2]\]
where $\zeta \sim \cN(0, I_p)$.  Then, we propose the update
\begin{equation}\label{eq:adj-SGLD}
\widetilde \theta_{k + 1, \eta} = \widetilde \theta_{k, \eta} - \eta_k \widehat{\nabla U}(\widetilde \theta_{k, \eta}) + \sqrt{2} \widehat{c}_k(\widetilde \theta_{k, \eta})\xi_k,    
\end{equation}
where $\widehat{c}_k(\theta)$ is assumed to be positive definite such that 
\[\eta^2_k \widehat{s}(\theta) \widehat{s}(\theta)^\top + 2\widehat{c}_k(\theta)\widehat{c}_k(\theta)^\top = 2\tau \eta_k I_p.\]
An analogous calculation shows that~\eqref{eq:adj-SGLD} can also be expressed by 
\begin{align}\label{eq:adj-update}
\begin{split}
\widetilde \theta_{k+1, \eta}  =\;&  \widetilde \theta_{k, \eta} - \eta_k \nabla U(\widetilde \theta_{k, \eta})  \\
& - \eta_k \varphi(\widetilde \theta_{k, \eta}) \zeta_k + \sqrt{2\tau\eta_k}\xi_k.
\end{split}
\end{align} 

To extend this method to the multi-variance replica exchange SGLD, we assume we have access to the following independent estimators, 
\begin{align*}
\widehat U_1(\theta) &\sim \cN(U(\theta), \sigma_1^2(\theta)), &\widehat U_2(\theta) &\sim \cN(U(\theta), \sigma_2^2(\theta)),\\
\widehat{\nabla U}_1(\theta) &\sim \cN(\nabla U(\theta), s_1(\theta)s_1(\theta)^\top), &\widehat{\nabla U}_2(\theta) &\sim \cN(\nabla U(\theta), s_2(\theta)s_2(\theta)^\top). 
\end{align*}
Further, assume we have estimators $\hat{s}_1, \hat{s}_2$ and  positive definite matrix $\widehat{c}_{i,k}(\theta)$ such that 
\begin{align*}
\eta_k^2 \widehat{s}_i(\theta)\widehat{s}_i(\theta)^\top +  2 \widehat{c}_{i,k}(\theta)\widehat{c}_{i,k}(\theta)^\top = \tau \eta_k I_p  
\end{align*}
for $i = 1, 2$.  We update the two chains by 
\begin{align}\label{eq:adj-multi-reSGLD}
\begin{split}
\widetilde \theta_{k + 1, \eta}^{(1)} &= \widetilde \theta_{k, \eta}^{(1)} - \eta_k \widehat{\nabla U}_1(\widetilde \theta_{k, \eta}^{(1)}) + \sqrt{2} \widehat c_{1,k}(\widetilde \theta_{k, \eta}^{(1)})\xi_k^{(1)}\\
\widetilde \theta_{k + 1, \eta}^{(2)} &= \widetilde \theta_{k, \eta}^{(2)} - \eta_k \widehat{\nabla U}_2(\widetilde \theta_{k, \eta}^{(2)}) + \sqrt{2} \widehat c_{2,k}(\widetilde \theta_{k, \eta}^{(2)}) \xi_k^{(2)},
\end{split}
\end{align} 
Additionally, instead of using \eqref{eq:multi-reSGLD-swap}, we define a new swapping rate \\
$a\eta_k\min\{1, \widetilde S(\widetilde \theta^{(1)}_{k, \eta}, \widetilde \theta^{(2)}_{k, \eta})\}$, where
\begin{align}\label{eq:adj-multi-reSGLD-swap}
\begin{split}
\widetilde S(\theta^{(1)}, \theta^{(2)}) \coloneqq \, &  \exp \Bigl\{  \tau_{\delta} \Bigl(\widehat U_1(\theta^{(1)}) - \widehat U_2(\theta^{(2)}) \\
&\qquad - \tau_{\delta}\left(\sigma_1^2(\theta^{(1)}) + \sigma_2^2(\theta^{(2)})\right) / 2\Bigr)\Bigr\}. 
\end{split}
\end{align}
Observe that when $\sigma^2 = \sigma_1^2(\theta^{(1)}) = \sigma_2^2(\theta^{(2)})$, \eqref{eq:adj-multi-reSGLD-swap} is reduced to the equal variance case \eqref{eq:reSGLD-swap}. Compared with \eqref{eq:multi-reSGLD-swap}, the main advantage of the new swapping rate is that we only need to calculate $U_1$ and $U_2$ once in \eqref{eq:adj-multi-reSGLD-swap}. 
This can significantly save computational time in problems where calculating the energy function $U$ is extremely time-consuming (e.g. in the inverse PDE problems presented in Section~\ref{sec:pde}).   
\begin{lemma}\label{thm:swap-estimator}
The estimator \eqref{eq:adj-multi-reSGLD-swap} satisfies 
$\bbE[\widetilde S(\theta^{(1)}, \theta^{(2)})  ] =  S(\theta^{(1)}, \theta^{(2)})$, where
 $S(\theta^{(1)}, \theta^{(2)})$ is given by~\eqref{eq:reSGLD-swap}.
\end{lemma}
\begin{proof} 
A routine calculation  using $\bbE[ e^{b Z }] = e^{b^2 / 2}$ for $Z \sim \cN(0, 1)$ and 
the independence between $\widehat{U}_1(\theta^{(1)})$, $\widehat{U}_2(\theta^{(2)})$ yields the result. 
\end{proof} 

\section{Convergence Analysis}
\label{sec_convergence}
In this section, we would like to introduce our theoretical results. We need more definitions and assumptions for later use to prove the convergence results. For two probability measures $\mu$ and $\nu$, the $2$-Wasserstein distance between $\mu$ and $\nu$ is defined by
\[W_2(\mu, \nu) = \left(\inf_{\gamma \in \Gamma(\mu, \nu)} \int \|\theta_{\mu} - \theta_{\nu}\|_2^2\, d\gamma(\theta_{\mu}, \theta_{\nu})\right)^{1/2},\]
where $\Gamma(\mu, \nu)$ is the coupling space, which includes all joint distributions $\gamma$ having $\mu$ and $\nu$ as marginal distributions. For any function $f$ that is continuously differentiable, we define the Dirichlet form as
\begin{equation}\label{eq:diri-form}
\bbD(U) \coloneqq \int \tau_1 \|\nabla_{\theta^{(1)}} f\|_2^2 + \tau_2 \|\nabla_{\theta^{(2)}} f\|_2^2 \, d\pi_{\rm{re}}(\theta^{(1)}, \theta^{(2)}),   
\end{equation}
where $\pi_{\rm{re}}$ is the stationary distribution defined in \eqref{eq:stationary-dis}. Moreover, define
\begin{align}\label{eq:diri-form-stationary}
\begin{split}
&\bbD_S(U) \coloneqq \bbD(U) \\
& + \frac{a}{2} \int S(\theta^{(1)}, \theta^{(2)}) \left(f(\theta^{(2)}, \theta^{(1)}) - f(\theta^{(1)}, \theta^{(2)})\right)^2 \, d\pi_{\rm{re}}(\theta^{(1)}, \theta^{(2)}).
\end{split}
\end{align}
For any two probability measures $\mu$ and $\nu$, we denote the relative entropy by
\begin{equation}\label{eq:rel-entropy}
\cE(\mu | \nu) = \int \log\frac{d\mu}{d\nu}\,d\mu.
\end{equation}
Following are some assumptions we need for the energy function and its derivative. 
\begin{enumerate}
\item Dissipativity. \label{cond3} The function $U$ is $(\alpha, \beta)$-dissipativity for $\alpha > 0$ and  $\beta \geq 0$, that is
\[\langle \theta, \nabla U(\theta) \rangle \geq \alpha \|\theta\|_2^2 - \beta\]
for all $\theta\in\bbR^p$.

\item Smoothness. \label{cond2} The gradient of function $U$ is $M$-Lipschitz continuous, that is
\[\|\nabla U(\theta) - \nabla U(\theta')\|_2 \leq M\|\theta - \theta'\|_2\]
for all $\theta, \theta'\in\bbR^p$.
\end{enumerate}
These two conditions are widely used in the theoretic analysis of the SGLD \citep{raginsky2017non, chen2020accelerating, deng2020non}. Now we are ready to state our first convergence results.

\begin{theorem}\label{thm:nonconvex-results}
Let $\widetilde \beta_{k, \eta} = (\widetilde \theta_{k, \eta}^{(1)}, \widetilde \theta_{k, \eta}^{(2)})$ be obtained from \eqref{eq:adj-multi-reSGLD} and $\mu_k$ be distribution of it. 
Denote $\varphi_k^{(1)} \coloneqq \varphi(\widetilde \theta^{(1)}_{k, \eta})$ and $\varphi_k^{(2)} \coloneqq \varphi(\widetilde \theta^{(2)}_{k, \eta})$, and the blocker matrix $\varphi_k \coloneqq \diag(\varphi_k^{(1)}, \varphi_k^{(2)})$. Let 
\[\psi_k = \widetilde S(\widetilde \theta_{k, \eta}^{(1)}, \widetilde \theta_{k, \eta}^{(2)}) - S(\widetilde \theta_{k, \eta}^{(1)}, \widetilde \theta_{k, \eta}^{(2)})\]
be the error due to the estimation of the swap rate. In addition, let $\pi_{\rm{re}}$ be the stationary distribution of reLD defined in \eqref{eq:stationary-dis}.
Given the dissipativity \ref{cond3} and smoothness \ref{cond2} assumptions, we have
\begin{align*}
W_2(\mu_k, \pi_{\rm{re}}) &\leq \widetilde C_0 \exp\left\{-\frac{k\eta(1 + C_{\bbD})}{C_{\cE}}\right\} \\
&\quad  + \widetilde C_{\tau_1, \tau_2, a, M, p}\biggl(\eta^{1/2} + \max_i \trace(\varphi_i\varphi_i^\top)^{1 / 2} + \max_i(\bbE[|\psi_i|^2])^{1 / 2}\biggr),    
\end{align*}
where $C_{\cE} > 0$ is constant, $\widetilde C_0 = \sqrt{2 C_{\cE} \cE(\mu_0 | \pi)}$, \[C_{\bbD} = \inf_{t > 0}\frac{\bbD\left(\frac{d\nu_k}{d\pi}\right)}{\bbD_S\left(\frac{d\nu_k}{d\pi}\right)} - 1\]
is a non-negative constant depending on the swapping rate $S$ and $\widetilde C_{\tau_1, \tau_2, a, M, p}$ is a constant depends on $\tau_1$, $\tau_2$, $a$, $M$ and dimension $p$.
\end{theorem}
\begin{proof}
The idea is by replacing $\phi_k$ in Theorem 1 of \citet{deng2020non} with $\varphi_k\boldsymbol{\zeta}_k$, where $\boldsymbol{\zeta}_k \sim \cN(0, I_{2p})$. We give proof details in \ref{appe:nonconvex}.
\end{proof}
Compared our Theorem \ref{thm:nonconvex-results} with Theorem 1 in \citet{deng2020non}, the error due to the gradient estimation, which is the dominant error \citep{teh2016consistency}, is significantly reduced. More specifically, denote $s_k^{(1)} \coloneqq s(\widetilde \theta^{(1)}_{k, \eta})$ and $s_k^{(2)} \coloneqq s(\widetilde \theta^{(2)}_{k, \eta})$, and the blocker matrix $s_k = \diag(s_k^{(1)}, s_k^{(2)})$. Under our setting that $\widehat{\nabla U}(\beta) \sim \cN(\nabla U(\beta), s(\beta)s(\beta)^\top)$, the error due to the estimation of gradient in \citet{deng2020non} was 
\[\max_i \trace\left(s_i s_i^\top\right)^{1/2}\]
while in our method is 
\[\max_i \trace\left(\varphi_i\varphi_i^\top\right)^{1/2}.\]
By \eqref{eq:var-est}, our method greatly reduces the error due to the estimation of the gradient.

The replica exchange method is typically employed to solve non-convex problems with multiple local modes. However, it can also be utilized to solve convex issues, especially when an adequate initialization cannot be supplied. We also provide a result of convergence when the energy function $U(\theta)$ is strongly convex.

\begin{enumerate}\addtocounter{enumi}{2}
\item Strongly convex assumption. \label{cond1} The energy function $U$ is $m$-strongly convex, that is 
\[U(\theta) - U(\theta') - \nabla U(\theta)^{\T} (\theta - \theta') \geq \frac{m}{2}\|\theta - \theta'\|_2^2\]
for all $\theta, \theta'\in\bbR^p$.
\end{enumerate}

When the energy function is strongly convex, we have the following estimation.
\begin{theorem}\label{thm:convex-results}
Let $\widetilde \beta_{k, \eta}$, $\varphi_k$, $\psi_k$ and $\pi_{\mathrm{re}}$ be the same as Theorem \ref{thm:nonconvex-results}.
Under the strongly convex \ref{cond1} and smoothness \ref{cond2} assumptions,  if $\eta < 1 / (m + M)$, we have
\begin{align*}
W_2(\mu_k, \pi_{\rm{re}}) & \leq  (1 - m\eta)^k W_2(\mu_0, \pi) \\
&\quad + C_{\tau_1, \tau_2, a, m, M}\biggl\{\eta \max_i \trace(\varphi_i\varphi_i^\top) + \max_i p\bbE[|\psi_i|] + \eta p\biggr\}^{1/2},    
\end{align*}
where $C_{\tau_1, \tau_2, a, m, M}$ is a constant depends on $\tau_1$, $\tau_2$, $a$, $m$ and $M$.
\end{theorem}
\begin{proof}
We defer the proof to \ref{appe:convex}.
\end{proof}
Compared to Theorem \ref{thm:nonconvex-results} under the non-convex setting, thanks to strongly convex, in Theorem \ref{thm:convex-results}, as the step size $\eta$ goes to $0$, the bias due to the estimation of the gradient will vanish, only the bias due to the estimation of swapping rate is left. The main idea of the proof is similar to Theorem 4 in \cite{dalalyan2019user}, but we reduce the error due to the estimation of gradient from $\max_i \trace\left(s_i s_i^\top\right)^{1/2}$ to $\max_i \trace\left(\varphi_i\varphi_i^\top\right)^{1/2}$.

\section{Experiments}\label{sec_exp}
We would like to give several numeric examples to show the benefits of our method.

\subsection{Gaussian mixture distribution simulations}
In this section, we evaluate our method for Gaussian mixture distribution. Let high temperature $\tau_2 = 10$ and low temperature $\tau_1 = 1$, then the samples sampled from low temperature follow the target distribution that we need. Consider the Gaussian mixture distribution
\begin{equation}\label{eq:gauss-mixt}
\exp(-U(\theta)) = 0.4\cN(\theta;-4, 0.7^2) + 0.6  \cN(\theta;3, 0.5^2).    
\end{equation}
We assume that in the low-temperature chain, we can only obtain an unbiased estimator $\widehat U_1(\theta) \sim \cN(U(\theta), 1^2)$, and in the high-temperature chain, we can only access $\widehat U_2(\theta) \sim \cN(U(\theta), 3^2)$. For the gradient, we also only access the noised version, where $ \widehat{\nabla U}_1(\theta)\sim \cN(\nabla U(\theta), 2^2)$ for the low temperature and $\widehat{\nabla U}_2(\theta)\sim \cN(\nabla U(\theta), 5^2)$ for the high temperature. We fixed the step size as $0.03$ in our updating step. Since we don't know the variance value in practice, to implement our method, we use the same method in \citet{deng2020non} to estimate the variance. More specifically, in each state $\widetilde \theta_{k, \eta}^{(l)}$, we get the sample variance $\widehat s^2(\widetilde \theta_{k, \eta}^{(l)})$ and update our variance estimator by
\[\widehat s_k^2 = (1 - 1 / k) \widehat s_{k - 1}^2 + (1 / k) \widehat s^2(\widetilde \theta_{k, \eta}^{(l)}),\]
where $\widehat s_k^2$ is our variance estimator in the $k$-th step, and $l$ can be $1$ for the low temperature and $2$ for the high-temperature samples. The same method was used to estimate $\sigma_1^2$ and $\sigma_2^2$, the variance of the estimator of the energy function.

\begin{figure}
\vspace{.3in}
    \centering
    \includegraphics[scale = 0.45]{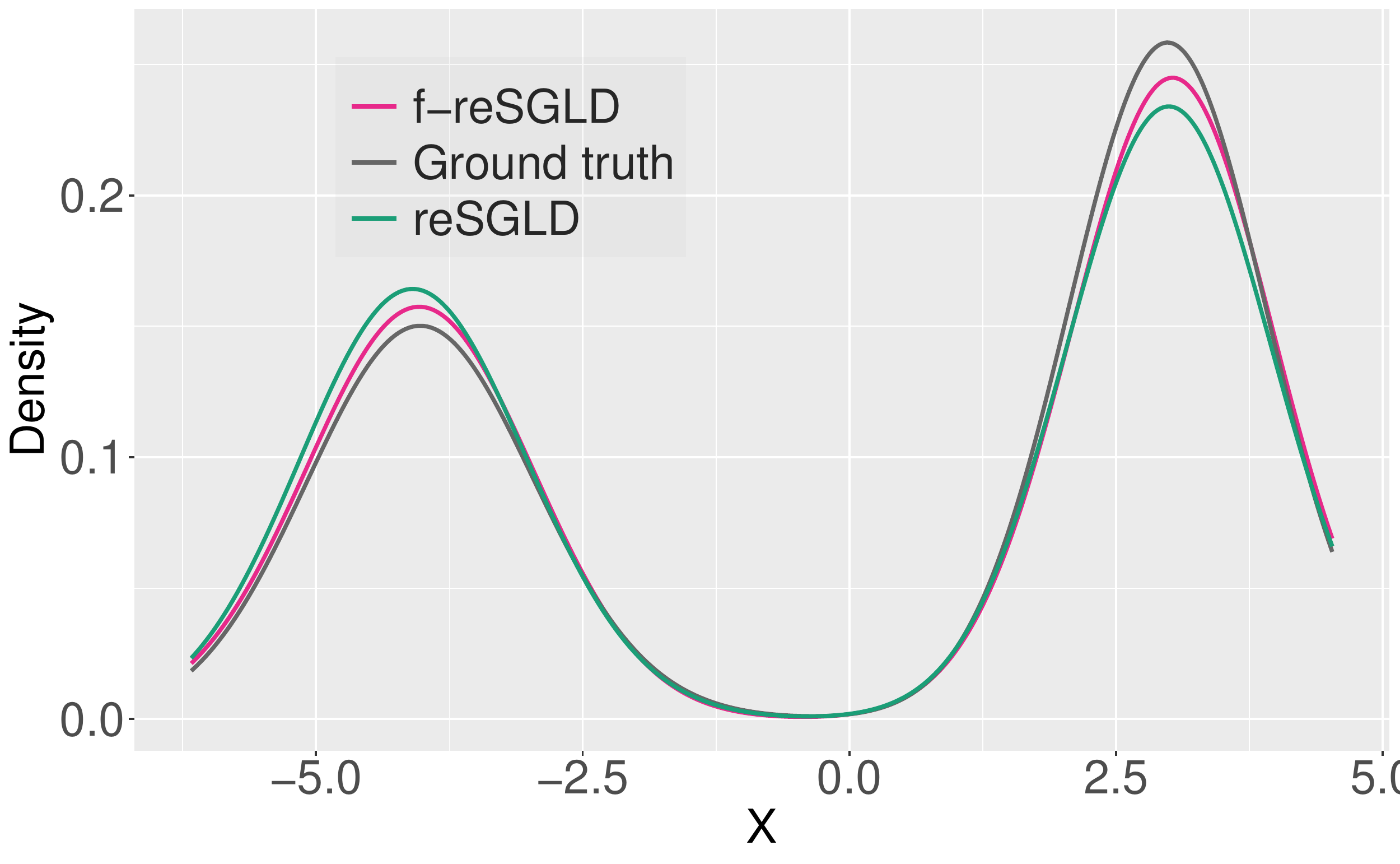}
\vspace{.3in}
    \caption{Gaussian mixture example where the variance of observation noise is fixed.}
    \label{fig:gauss-mix-fix}
\end{figure}

Figure \ref{fig:gauss-mix-fix} shows the density plot of $1000$ samples in this setting, where the black line is the ground truth, the green line is from the basic reSGLD method, and the red line is from our method. It is easy to see that our method fits better than the reSGLD method, which is closer to the ground truth. 

\begin{figure}
\vspace{.3in}
    \centering
    \includegraphics[scale = 0.45]{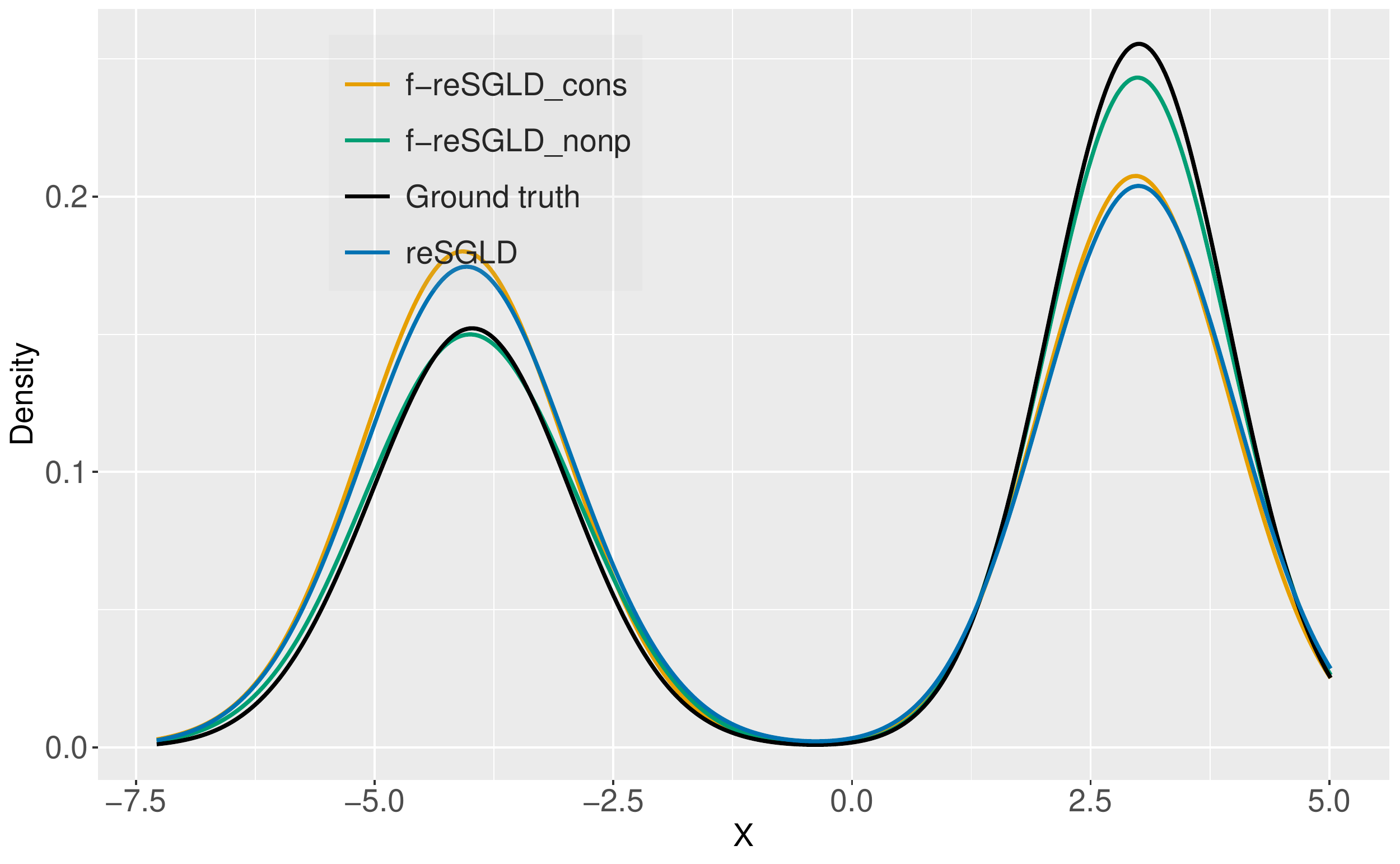}
\vspace{.3in}
    \caption{Gaussian mixture example where the variance of observation noise is fixed.}
    \label{fig:gauss-mix-flex}
\end{figure}

Under the same Gaussian mixture distribution given in \eqref{eq:gauss-mixt}, we next consider a more general case that the variance of error $s^2(\theta)$ depends on the state $\theta$. In our simulation, we assume
\[s(\theta) = \frac{5\exp(\theta)}{1 + \exp(\theta)}, \text{and }\sigma(\theta) = \frac{3\exp(U(x))}{2(1 + \exp(U(x)))},\]
where $\widehat U(\theta) \sim \cN(U(\theta), \sigma^2(\theta))$ and $\widehat{\nabla U(\theta)} \sim \cN(\nabla U(\theta), s^2(\theta))$.We tried two different methods to estimate the variance. In the first method, we assume that the variance is a constant, and we use the same method in \citet{deng2020non} described above. We refer to this method as ``f-reSGLD\_cons''. In the second method, we used a non-parametric method to get an estimator of variance,  which is denoted as ``f-reSGLD\_nonp''. More specifically, we collected first $100$ samples $\{\widetilde \theta_i^{(2)}\}_{i = 1}^{100}$ in the high-temperature chain with their sample variance $\{\widehat s^2(\widetilde \theta_i^{(2)})\}_{i = 1}^{100}$. Next, we fit a kernel ridge regression (KRR) \citet{wainwright2019high} to get an estimator $\widehat \cF$ of variance. Then, in the following updates, for each state $\widetilde \theta$, we calculate its variance via $\widehat \cF(\widetilde \theta)$.

Figure \ref{fig:gauss-mix-flex} shows results for the non-constant variance case, where the black line is the ground truth, the blue line is the standard reSGLD, the yellow line is our method where we estimate variance as the constant, and the green line is f-reSGLD method where estimating variance via KRR. The performance of reSGLD and f-reSGLD\_cons is similar, while f-reSGLD\_nonp performs much better than those two, which is closer to the ground truth. This implies that the better estimator of the variance we get, the better the performance of our method.

\subsection{Inverse PDE}\label{sec:pde}
We next present an inverse PDE example. This example is to demonstrate that the proposed sampling method can capture multi-mode inverse quantities of interest (iQoI) with large noise in the likelihood functions and the gradient estimations. In particular, we show that the proposed method can reach the same accuracy and effectiveness (the number of samplings) as reSGLD with much lower injected noise, while the method without the noise correction cannot capture all iQoI.

We design the problem so that there are an infinite number of iQoI. The inverse PDE relies on the following model equation,
\begin{align*}
    & u_t = \nabla\cdot \nabla u +f, x\in \Omega = [0, 1]^2,, t\in[0, T]\\
    & u(x, t) = 0, x\in\partial\Omega,\\
    & u(x, 0) = \beta e^{-(x-x_0)^2/\alpha}.
\end{align*}
The initial condition is unknown, or $x_0$ is unknown. 
The target is to track $x_0$ given the measurement $u(x, t)$ of a single sensor at the terminal time $T$ and location $x_s$.
We set the exact solution to $u(x, t) = \beta e^{-(x-x_0)^2/\alpha}e^{-t}$, if we place only one sensor at $x_s$, the iQoI will be a circle centered in $x_s$.
In this work, $\beta = 1/(2\pi h^2)$, $\alpha = 2h^2$, where $h = 0.1$. $T = 0.03$ and the sensor is placed at $(0.3, 0.5)$.

We perform three sets of experiments: vanilla reSGLD with small noise (s-reSGLD), fast reSGLD (f-reSGLD) with large noise, and vanilla reSGLD with large noise (l-reSGLD).
We manually inject noise into the likelihood functions and the gradient estimations. Specifically, we inject a Gaussian noise $\mathcal{N}(0, 0.1^2)$ to the likelihood functions of the s-reSGLD while injecting a Gaussian noise $\mathcal{N}(0, 0.8^2)$ to the likelihood functions of 
f-reSGLD and l-reSGLD. For the gradient of the likelihood functions, we inject a Gaussian noise $\mathcal{N}(0, 0.1^2)$ for s-reSGLD while injecting a Gaussian noise $\mathcal{N}(0, 2^2)$ for f-reSGLD and l-reSGLD. In addition, in all three experiments, the low temperature is equal to 0.08 and the high temperature is equal to 0.5, the effective temperatures of the proposed method are then derived accordingly. The results are shown in Figures \ref{exp1_12000} and \ref{exp1_48000}.

\begin{figure}
\centering
\includegraphics[scale = 0.35]{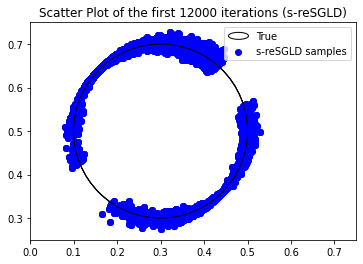}
\includegraphics[scale =
0.35]{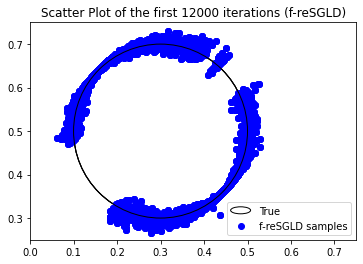}
\includegraphics[scale =
0.35]
{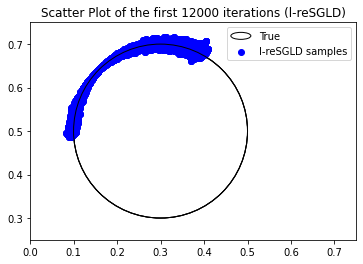}
\caption{The first 12,000 samples of all three methods. Left: reSGLD with small injected noise (s-reSGLD). Middle: the proposed method with large injected noise (f-reSGLD). Right: reSGLD with the same large injected noise (l-reSGLD) as the proposed method (f-reSGLD).}
\label{exp1_12000}
\end{figure}

\begin{figure}
\centering
\includegraphics[scale = 0.35]{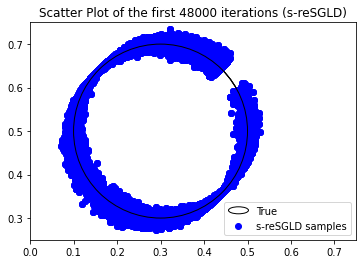}
\includegraphics[scale =
0.35]{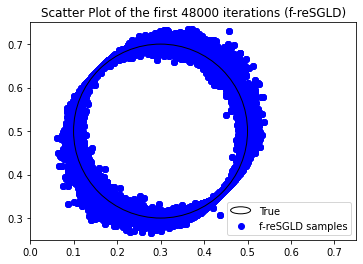}
\includegraphics[scale =
0.35]{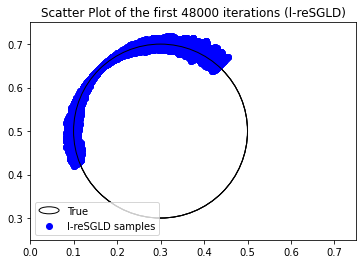}
\caption{The first 48,000 samples of all three methods. Left: reSGLD with small injected noise (s-reSGLD). Middle: the proposed method with large injected noise (f-reSGLD). Right: reSGLD with the same large injected noise (l-reSGLD) as the proposed method (f-reSGLD).}
\label{exp1_48000}
\end{figure}

From Figures \ref{exp1_12000} and \ref{exp1_48000}, we can observe that the proposed method with large noises in the likelihood function and the gradient of the likelihood function can capture the true iQoI, 
however, the vanilla reSGLD with the same noises cannot capture all iQoI.

\section{Conclusion}\label{sec_con}
In this work, we propose a method that can lower the accuracy requirement of the energy function estimator while preserving convergence. More specifically, our estimation reduces the error of the gradient estimation compared with the vanilla reSGLD. In addition, we also present another estimation of convergence under the strongly convex assumption in Theorem~
\ref{thm:convex-results}. Compared to the non-convex setting, the gradient error will vanish as the learning rate goes to zero. Moreover, we propose a new swap rate estimator and prove it is unbiased. Compared to the vanilla swapping rate in m-reSGLD, our swapping rate does not require extra likelihood evaluation and is much faster. Several numerical examples show that the method can tolerate the noise in the gradient estimation and achieve the same convergence, but the vanilla m-reSGLD may not converge. Our future work will be on how to get a tighter bound when the convex assumption is removed.

\section*{Acknowledgments} 
GL and ZZ gratefully acknowledge the support of the National Science Foundation (DMS-1555072, DMS-2053746, and DMS-2134209), Brookhaven National Laboratory Subcontract 382247, and U.S. Department of Energy (DOE) Office of Science Advanced Scientific Computing Research program DE-SC0021142 and DE-SC0023161.

\bibliographystyle{elsarticle-harv} 
\bibliography{references}

\appendix
 
\section{Proof of Theorem \ref{thm:convex-results}}\label{appe:convex}
Since the swaps of the positions are equivalent to swaps of the temperatures \citep{dupuis2012infinite, chen2020accelerating}, we can model reLD by the following stochastic differential equation:
\[d\beta_t = -\nabla G(\beta_t)d t + \Sigma_t d W_t,\]
where $\beta_{t} = (\theta_{t}^{(1)}, \theta_{t}^{(2)})$, $G(\beta_t) = [U(\theta_t^{(1)}), U(\theta_t^{(2)})]$, $W_t\in\bbR^{2p}$ is Brownian motion, and $\Sigma_t$ is a random matrix that swaps between the diagonal matrices 
\[M_1 = \begin{bmatrix}\sqrt{2\tau_1 I_p} & 0 \\ 0 & \sqrt{2\tau_2I_p}\end{bmatrix} \text{ and }M_2 = \begin{bmatrix}\sqrt{2\tau_2I_p} & 0 \\ 0 & \sqrt{2\tau_1 I_p}\end{bmatrix}\]
with probability $a S(\theta_t^{(1)}, \theta_t^{(2)})dt$, $I_p$ is $p$-dimensional identity matrix.

Recall that our update procedure is 
\[\widetilde \beta_{k+1, \eta} = \widetilde \beta_{k, \eta} - \eta\nabla \widetilde G(\widetilde \beta_{k, \eta}) + \sqrt{\eta} \widetilde \Sigma_{k, \eta} \xi_k,\]
where $\widetilde G(\widetilde \beta_{k, \eta}) = [\widetilde U(\widetilde \theta_{k, \eta}^{(1)}), \widetilde U(\widetilde \theta_{k, \eta}^{(2)})]$ and $\widetilde \Sigma_{k, \eta}$ is a random matrix that swaps between diagonal matrix $M_1$ and $M_2$ with probability $a\eta\widetilde S(\widetilde \theta_{k, \eta}^{(1)}, \widetilde \theta_{k, \eta}^{(2)})$. Without loss of generality, we assume that our initialized point is $\widetilde \beta_{0, \eta} \sim \mu_0$. Recall that $\mu_k$ is the distribution of $\widetilde \beta_{k, \eta}$. To later use, we define the denoised update by
\[\beta_{k + 1, \eta} = \beta_{k, \eta} - \eta \nabla G(\beta_{k, \eta}) + \sqrt{\eta}\Sigma_{k, \eta}\xi_k,\]
where $\Sigma_{k, \eta}$ is a random matrix that swaps between $M_1$ and $M_2$ with probability $a\eta S(\theta^{(1)}_{k, \eta}, \theta^{(2)}_{k, \eta})$.

\begin{proof}
Let $L_0$ be a random vector drawn from $\pi_{\rm{re}}$ such that $W_2(\mu_0, \pi_{\rm{re}}) = \bbE[\|\widetilde \beta_{0, \eta} - L_0\|_2]$. We define the stochastic process $\{L_t\}_{t\geq 0}$ by
\[L_t = L_0 - \int_0^t \nabla G(L_s)\, ds + \int_0^t \Sigma_s\,dW_s.\]
Since $\pi_{\rm{re}}$ is the stationary distribution, we know that $L_t\sim\pi_{\rm{re}}$ for all $t \geq 0$. Observe that
\[L_{(k+1)\eta} = L_{k\eta} - \int_{k\eta}^{(k+1)\eta}\nabla G(L_s)\,ds + \int_{k\eta}^{(k+1)\eta}\Sigma_s\,dW_s.\]

Denote $\Delta_k = L_{k\eta} - \widetilde \beta_{k, \eta}$, then
\begin{align}\label{eq:decom-A}
\begin{split}
\Delta_{k+1} &= L_{k\eta} - \int_{k\eta}^{(k+1)\eta}\nabla G(L_s)\,ds + \int_{k\eta}^{(k+1)\eta}\Sigma_s\,dW_s \\
&\qquad - (\widetilde \beta_{k, \eta} - \eta\nabla \widetilde G(\widetilde \beta_{k, \eta}) + \sqrt{\eta} \widetilde \Sigma_{k, \eta} \xi_k)\\
&= \Delta_k + \eta\nabla \widetilde G(\widetilde \beta_{k, \eta})- \int_{k\eta}^{(k+1)\eta}\nabla G(L_s)\,ds + \int_{k\eta}^{(k+1)\eta}(\Sigma_s - \widetilde \Sigma_{k,\eta})\,dW_s\\
&= \Delta_k  - \eta(\underbrace{\nabla G(L_{k\eta}) - \nabla G(\widetilde \beta_{k, \eta})}_{\coloneqq A_1}) - \underbrace{\int_{k\eta}^{(k+1)\eta} (\nabla G(L_s) - \nabla G(L_{k\eta}))\,ds}_{\coloneqq A_2} \\
& \qquad  + \eta\underbrace{(\nabla \widetilde G(\widetilde \beta_{k, \eta}) - \nabla G(\widetilde \beta_{k, \eta}))}_{\coloneqq A_3} \\
& \qquad + \underbrace{\int_{k\eta}^{(k+1)\eta}(\Sigma_s - \Sigma_{k,\eta})\,dW_s}_{\coloneqq A_4} + \underbrace{\int_{k\eta}^{(k+1)\eta}(\Sigma_{k, \eta} - \widetilde \Sigma_{k,\eta})\,dW_s}_{\coloneqq A_5}.    
\end{split}
\end{align}

Notice that $A_3 = (\varphi_k^{(1)}\zeta_k^{(1)}, \varphi_k^{(2)}\zeta_k^{(2)})$, where $\zeta_k^{(1)}, \zeta_k^{(2)}\sim \cN(0, I_p)$. By the definition of $\varphi_k^{(1)}$ and $\varphi_k^{(2)}$, $A_3$ is independent with all other terms given $\widetilde \beta_{k, \eta}$ and $\bbE[A_3 |\widetilde \beta_{k, \eta}] = 0$. Furthermore, we have
\begin{equation}\label{eq:bound-A3}
\bbE[\|A_3\|_2^2] = \trace(\psi_k\psi_k^\top).
\end{equation}

Observe that $A_4$ is also independent of all other terms given $\widetilde \beta_{k, \eta}$, and by the definition of It\^{o} integral, we know $\bbE[A_4 | \widetilde \beta_{k, \eta}] = 0$. According to It\^{o} isometry, we have
\begin{align*}
&\bbE\left[\left\Vert \int_{k\eta}^{(k+1)\eta}(\Sigma_s - \Sigma_{k,\eta})\,dW_s\right\Vert_2^2\right]\\
=&\sum_{j = 1}^{2p}\int_{k\eta}^{(k+1)\eta}\bbE[(\Sigma_s(j) -  \Sigma_{k,\eta}(j))^2]\, ds\\
=& \sum_{j = 1}^{2p} 2(\sqrt{\tau_2} - \sqrt{\tau_1})^2 \int_{k\eta}^{(k+1)\eta}P(\Sigma_s(j) \neq \Sigma_{k, \eta}(j))\,ds.
\end{align*}
Following the discussions in \citet{chen2020accelerating, deng2020non}, by the conditional probability 
\[P(\Sigma_s(j) \neq \Sigma_{k, \eta}(j) | \widetilde \beta_{k, \eta}) = a S(\widetilde \theta_{k,\eta}^{(1)}, \widetilde \theta_{k,\eta}^{(1)})(s - k\eta) + a o(s - k\eta),\]
where $o(s - k\eta)$ is the higher remainder with respect to $s - k\eta$. Hence,
\begin{align}\label{eq:bound-A4}
\begin{split}
\bbE[\|A_4\|_2^2] &= \sum_{j = 1}^{2p} 2(\sqrt{\tau_2} - \sqrt{\tau_1})^2 a \left[\int_{k\eta}^{(k+1)\eta}(s - k\eta) + o(s - k\eta)\,ds\right]\\
& \leq \delta_1(\tau_1, \tau_2, a) p \eta^2,
\end{split}
\end{align}
where $\delta_1(\tau_1, \tau_2, a)$ is a constant depends on $\tau_1$, $\tau_2$ and $a$.

Since $A_5 = \sqrt{\eta}(\Sigma_{k, \eta} - \widetilde \Sigma_{k, \eta})\zeta$, where $\zeta\sim\cN(0, I_{2p})$, then $\bbE[A_5] = 0$. Moreover, the upper bound of $\bbE[\|A_5\|_2^2]$ can be gotten by
\begin{align}\label{eq:bound-A5}
\begin{split}
\bbE[\|A_5\|_2^2] &= \sum_{j = 1}^{2p} 2(\sqrt{\tau_2} - \sqrt{\tau_1}) \int_{k\eta}^{(k+1)\eta}P(\Sigma_{k, \eta}(j) \neq \widetilde \Sigma_{k, \eta}(j))\,ds\\
&= \sum_{j = 1}^{2p} 2(\sqrt{\tau_2} - \sqrt{\tau_1})^2 a\eta\bbE[|S(\widetilde \theta_{k, \eta}^{(1)}, \widetilde \theta_{k, \eta}^{(2)}) - \widetilde S(\widetilde \theta_{k, \eta}^{(1)}, \widetilde \theta_{k, \eta}^{(2)})|]\\
&\leq \delta_2(\tau_1, \tau_2, a) p \eta \bbE[|\psi_k|],
\end{split}
\end{align}
where $\delta_2(\tau_1, \tau_2, a)$ is a constant depends on $\tau_1$, $\tau_2$ and $a$. The bound \eqref{eq:bound-A5} was first gotten by \citet{deng2020non} and we present this again for the convenience of the reader.

Since the expectation of $A_3$ and $A_4$ is equal to 0 given $\widetilde \beta$, and $\bbE[A_5] = 0$, together with the results in \eqref{eq:bound-A3}, \eqref{eq:bound-A4} and \eqref{eq:bound-A5}, we have
\begin{align}\label{eq:decom2}
\begin{split}
\bbE[\|\Delta_{k + 1}\|_2^2] &= \bbE[\|\Delta_k - \eta A_1 - A_2 + \eta A_3 + A_4 + A_5\|_2^2] \\
& = \bbE[\|\Delta_k - \eta A_1 - A_2 \|_2^2] + \eta^2 \bbE[\|A_3\|_2^2] + \bbE[\|A_4\|_2^2] + \bbE[\|A_5\|_2^2]\\
& \leq \left(\bbE[\|\Delta_k - \eta A_1\|_2] + \bbE[\|A_2\|]\right)^2 \\
&\qquad + \eta^2\trace(\psi_k\psi_k^\top) + \delta_1 p\eta^2 + \delta_2 p \eta \bbE[|\psi_k|].
\end{split}
\end{align}


By Lemma 2 in \citet{dalalyan2019user}, since $U$ is $m$-strongly convex and $\nabla U$ is $M$-Lipschitz, and $\eta < 1 / (M + m)$, we have
\[\|\theta - \theta' - \eta(\nabla U(\theta) - \nabla U(\theta'))\|_2 \leq (1 - m\eta) \|\theta - \theta'\|_2.\]
Hence, in our case,
\begin{align*}
& \|\beta - \beta' - \eta(\nabla G(\beta) - \nabla G(\beta'))\|_2^2 \\
= &\|\theta^{(1)} - \theta'^{(1)} - \eta(\nabla U(\theta^{(1)}) - \nabla U(\theta'^{(1)}))\|_2^2\\
&\quad + \|\theta^{(2)} - \theta'^{(2)} - \eta(\nabla U(\theta^{(2)}) - \nabla U(\theta'^{(2)}))\|_2^2\\
\leq &(1 - m\eta)^2\|\theta^{(1)} - \theta'^{(1)}\|_2^2 + (1 - m\eta)^2\|\theta^{(2)} - \theta'^{(2)}\|_2^2\\
= &(1 - m\eta)^2 \|\beta - \beta'\|_2^2.
\end{align*}
Thus,
\begin{align}\label{eq:bound-delta}
\begin{split}
\|\Delta_k - \eta A_1\|_2 =& \|L_{k\eta} - \widetilde \theta_{k, \eta} - \eta(\nabla G(L_{kh}) - \nabla G(\widetilde \theta_{k, h}))\|_2 \\
& \leq (1 - m\eta) \|\Delta_k\|_2.      
\end{split}  
\end{align}

By Lemma 3 in \citet{dalalyan2019user}, since the function $U$ is continuously differentiable, and the gradient of $U$ is Lipschitz with constant $M$, then
\[\bbE[\|\nabla U(\theta)\|_2^2] \leq Mp.\]
For the term $A_2$, because all $L_t$ follows the same distribution, we only need to consider the case when $k = 0$. Then,
\begin{align*}
\bbE[\|A_2(0)\|_2] &\coloneqq \bbE\left[\left\Vert\int_0^{\eta}(\nabla G(L_s) - \nabla G(L_0))\,ds\right\Vert_2\right]\\
&\leq \int_0^{\eta}\bbE[\|\nabla G(L_s) - \nabla G(L_0)\|_2]\,ds\\
&\leq M\int_0^{\eta} \bbE[\|L_s - L_0\|_2]\,ds\\
&= M\int_0^{\eta}\bbE\left[\left\Vert - \int_0^s \nabla G(L_t)\, dt + \int_0^s \Sigma_t\, d W_t\right\Vert_2\right] \,ds\\
&\leq M\underbrace{\int_0^{\eta}\int_0^s \bbE[\|\nabla G(L_t)\|_2]\,dt\,ds}_{\coloneqq B_1} + M\underbrace{\int_0^{\eta}\bbE\left[\left\Vert\int_0^s \Sigma_t \,dW_t\right\Vert_2\right]\,ds}_{\coloneqq B_2}.
\end{align*}
For the term $B_1$, since $L_t$ follows distribution $\pi$ for all $t\geq 0$, due to Lemma 3,
\[\bbE[\|\nabla G(L_t)\|_2^2] = \bbE[\|\nabla G(L_0)\|_2^2] = \bbE[\|\nabla U(L_0^{(1)})\|_2^2] + \bbE[\|\nabla U(L_0^{(2)})\|_2^2] \leq 2 M p.\]
Hence,
\begin{equation}\label{eq:bound-B1}
B_1 \leq \int_0^{\eta}\int_0^s \sqrt{2Mp}\,dt\,ds = \frac{\sqrt{2Mp}\eta^2}{2}.
\end{equation}
Since $\Sigma_t$ is a diagonal matrix with elements $\sqrt{2\tau_1}$ or $\sqrt{2\tau_2}$ for any $t$, we have
\begin{align}\label{eq:bound-B2}
\begin{split}
B_2 \leq \int_0^{\eta} \sqrt{2\tau_2} \bbE[\|W_s\|_2]\,ds \leq \sqrt{2\tau_2} \int_0^{\eta} \sqrt{2ps}\,ds = \frac{4\sqrt{p\tau_2}\eta^{3/2}}{3}.
\end{split}
\end{align}
Combine \eqref{eq:bound-B1} and \eqref{eq:bound-B2}, we have
\[\bbE[\|A_2\|_2] \leq \frac{\sqrt{2}M^{3/2}p^{1/2}\eta^2}{2} + \frac{4Mp^{1/2}\eta^{3/2}}{3}.\]
Due to the assumption that $\eta < 1 / M$, we have
\begin{equation}\label{eq:bound-A2}
\bbE[\|A_2\|_2] \leq \delta_3 M \eta^{3/2} p^{1/2}.   
\end{equation}

Plug \eqref{eq:bound-delta} and \eqref{eq:bound-A2} into \eqref{eq:decom2}, we have
\begin{align*}
\bbE[\|\Delta_{k + 1}\|_2^2] &\leq \left((1 - m\eta)\bbE[\|\Delta_k\|_2] + \delta_3 M \eta^{3/2} p^{1/2} \right)^2 \\
&\quad + \eta^2\max_{i}\trace(\varphi_i\varphi_i^\top) + \delta_1 p\eta^2 + \delta_2 p \eta \max_i \bbE[|\psi_i|].
\end{align*}
Finally, due to Lemma 1 in \citet{dalalyan2019user}, we have
\begin{align*}
&W_2(\mu_k, \pi_{\rm{re}}) \\
\leq &(1 - m\eta)^k W_2(\mu_0, \pi_{\rm{re}})  + \frac{\delta_3 M \eta^{3/2}p^{1/2}}{m\eta} \\
& + \frac{\eta^2\max_{i}\trace(\varphi_i\varphi_i^\top) + \delta_1 p\eta^2 + \delta_2 p \eta \max_i \bbE[|\psi_i|]}{\delta_3 M \eta^{3/2}p^{1/2} + \sqrt{m\eta\left(\eta^2\max_{i}\trace(\varphi_i\varphi_i^\top) + \delta_1 p\eta^2 + \delta_2 p \eta \max_i \bbE[|\psi_i|]\right)}}\\
\leq &(1 - m\eta)^k W_2(\mu_0, \pi_{\rm{re}}) \\
& + C_{\tau_1, \tau_2, a, m, M}\left(\eta\max_i\trace(\varphi_i\varphi_i^\top) + p\eta + p\max_i\bbE[\psi|]\right)^{1/2}.
\end{align*}
\end{proof}

\section{Proof of Theorem \ref{thm:nonconvex-results}}\label{appe:nonconvex}
\begin{proof}
Define
\[\beta_t = \beta_0 - \int_0^t \nabla G(\beta_s)\,ds + \int_0^t \Sigma_s d W_s,\]
where $\beta_t = (\theta_t^{(1)}, \theta_t^{(2)})$, $G(\beta_t) = [U(\theta_t^{(1)}), U(\theta_t^{(2)})]$, $W_t\in\bbR^{2p}$ is Brownian motion, and $\Sigma_t$ is a random matrix that swaps between the diagonal matrices 
\[M_1 = \begin{bmatrix}\sqrt{2\tau_1 I_p} & 0 \\ 0 & \sqrt{2\tau_2 I_p}\end{bmatrix} \text{ and }M_2 = \begin{bmatrix}\sqrt{2\tau_2 I_p} & 0 \\ 0 & \sqrt{2\tau_1 I_p}\end{bmatrix}\]
with probability $a S(\theta_t^{(1)}, \theta_t^{(2)})dt$, $ I_p$ is $p$-dimensional identity matrix. Define $\beta_{k, \eta}\coloneqq \beta_{k\eta}$ and denote the distribution of $\beta_{k, \eta}$ as $\nu_k$. Notice that we can revise our update \eqref{eq:adj-multi-reSGLD} as 
\[\widetilde \beta_{k, \eta} = \widetilde \beta_0 - \int_0^{k\eta} \nabla \widetilde G(\widetilde \beta_{\lfloor s / \eta \rfloor, \eta})\,ds + \int_0^{k\eta} \widetilde \Sigma_{\lfloor s / \eta \rfloor, \eta} d W_s,\]
where $\lfloor x \rfloor$ is the greatest integer less than or equal to $x$. Denote the distribution of $\widetilde \beta_{k, \eta}$ as $\mu_k$. To later use, we define $\{\widetilde \beta^{\eta}_t\}$ as the continuous-time interpolation of $\{\widetilde \beta_{k, \eta}\}$, which is a continuous-time stochastic process defined by 
\[\widetilde \beta^{\eta}_t = \widetilde \beta_0 - \int_0^t \nabla \widetilde G(\widetilde \beta_{\lfloor s / \eta \rfloor, \eta})\,ds + \int_0^t \widetilde \Sigma_{\lfloor s / \eta \rfloor, \eta} d W_s.\]
By the triangle inequality, we have
\[W_2(\mu_k, \pi_{\rm{re}}) \leq W_2(\mu_k, \nu_k) + W_2(\nu_k, \pi_{\rm{re}}).\]

For the second term, by Lemma 5 in \citet{deng2020non}, given dissipativity \ref{cond3} and smoothness \ref{cond2} assumptions, we have
\begin{equation}\label{eq:bound-nu}
W_2(\nu_k, \pi_{\rm{re}}) \leq  C_0 \exp\left\{-\frac{k\eta(1 + C_{\bbD})}{C_{\cE}}\right\},
\end{equation}
where $C_{\cE} > 0$ is constant, $C_0 = \sqrt{2 C_{\cE} \cE(\nu_0 | \pi)}$ and \[C_{\bbD} = \inf_{t > 0}\frac{\bbD\left(\frac{d\nu_k}{d\pi}\right)}{\bbD_S\left(\frac{d\nu_k}{d\pi}\right)} - 1\]
is a non-negative constant depending on the swapping rate $S$ defined in \eqref{eq:swap-reLD}. 

Next is to bound the discretization error term $W_2(\mu_k, \nu_k)$. Let's assume $\beta_0 = \widetilde \beta_0$. Then,
\[\beta_t - \widetilde \beta^{\eta}_t = -\int_0^t \left(\nabla G(\beta_s) - \nabla \widetilde G(\widetilde \beta_{\lfloor s / \eta \rfloor, \eta})\right)\, ds + \int_0^t (\Sigma_s - \widetilde \Sigma_{\lfloor s / \eta \rfloor, \eta})\,d W_s.\]

By the triangle inequality and Minkowski inequality, we have
\begin{align}\label{eq:decom-I}
\begin{split}
\bbE[\|\beta_t - \widetilde \beta^{\eta}_t\|_2^2]^{1/2} \leq & \bbE\left[\underbrace{\left\Vert \int_0^t \left(\nabla G(\beta_s) - \nabla \widetilde G(\widetilde \beta_{\lfloor s / \eta \rfloor, \eta})\right)\, ds \right\Vert_2^2}_{\coloneqq \cI_1}\right]^{1/2} \\
& \qquad + \bbE\left[\underbrace{\left\Vert \int_0^t (\Sigma_s - \widetilde \Sigma_{\lfloor s / \eta \rfloor, \eta})\,d W_s \right\Vert_2^2}_{\coloneqq \cI_2}\right]^{1/2}.    
\end{split}
\end{align}

Let's focus on $\cI_1$ first. Notice that
\begin{align}\label{eq:decom-I1}
\begin{split}
\bbE[\cI_1] & \leq \bbE\left[t \int_0^{t}\|\nabla G(\beta_s) - \nabla \widetilde G(\widetilde \beta_{\lfloor s / \eta \rfloor, \eta})\|_2^2 \, ds\right]\\
&\leq 3 t \left\{\bbE\left[\underbrace{\int_{0}^{t}\|\nabla G(\beta_s) - \nabla G(\widetilde \beta_{\eta}^s)\|_2^2\,ds}_{\coloneqq \cI_{1a}}\right]\right. \\
& \qquad + \bbE\left[\underbrace{\int_{0}^{t}\|\nabla G(\widetilde \beta_{\eta}^s) - \nabla G(\widetilde \beta_{\lfloor s / \eta \rfloor, \eta})\|_2^2\,ds}_{\coloneqq \cI_{1b}}\right]\\
& \qquad + \left.\bbE\left[\underbrace{\int_{0}^{t}\|\nabla G(\widetilde \beta_{\lfloor s / \eta \rfloor, \eta}) - \nabla \widetilde G(\widetilde \beta_{\lfloor s / \eta \rfloor, \eta})\|_2^2\,ds}_{\coloneqq \cI_{1c}}\right]\right\}.   
\end{split}
\end{align}
By the smoothness assumption \ref{cond2},
\begin{equation}\label{eq:bound-I1a}
\bbE[\cI_{1a}] \leq M^2 \bbE\left[\int_{0}^t\|\beta_s - \widetilde \beta^s_{\eta}\|_2^2\, ds\right].
\end{equation}
Denote $\kappa = \lfloor t / \eta \rfloor$, by the smoothness assumption \ref{cond2},
\begin{align}\label{eq:decom-I1b}
\begin{split}
&\bbE[\cI_{1b}] \\
\leq& M^2 \sum_{i = 0}^{\kappa} \bbE\left[\int_{i\eta}^{(i+1)\eta}\|\widetilde \beta^s_{\eta} - \widetilde \beta_{i, \eta}\|_2^2\,ds\right]\\
= &M^2 \sum_{i = 0}^{\kappa} \bbE\left[\int_{i\eta}^{(i+1)\eta}\left\Vert-\nabla \widetilde G(\widetilde \beta_{i, \eta})(s - i\eta) + \widetilde \Sigma_{i, \eta}\int_{i\eta}^s\,d W_t\right\Vert_2^2\,ds \right]\\
\leq &2M^2 \sum_{i = 1}^{\kappa} \left\{\int_{i\eta}^{(i+1)\eta} (s - i\eta)^2 \bbE[\|\nabla \widetilde G(\widetilde \beta_{i, \eta})\|_2^2]\,ds + \int_{i\eta}^{(i+1)\eta}\bbE\left[\left\Vert\widetilde \Sigma_{i, \eta}\int_{i\eta}^s\,d W_t\right\Vert_2^2\right]\,ds \right\}    
\end{split}
\end{align}
Observe that
\begin{align*}
\bbE[\|\nabla \widetilde G(\widetilde \beta_{i, \eta})\|_2^2] &= \bbE[\|\nabla G(\widetilde \beta_{i, \eta}) + \varphi_i \zeta_i\|_2^2]\\
&\leq 2\bbE [\|\nabla G(\widetilde \beta_{i, \eta})\|_2^2] + 4 \trace(\varphi_i\varphi_i^\top)]\\
& = 2 \bbE[\|\nabla G(\widetilde \beta_{i, \eta}) - \nabla G(\beta^*)\|_2^2] + 4 \trace(\varphi_i\varphi_i^\top)]\\
& \leq 2M^2\bbE[\|\widetilde \beta_{i, \eta} - \beta^*\|_2^2] + 4 \trace(\varphi_i\varphi_i^\top)]\\
& \leq 4M^2 \bbE[\|\widetilde \beta_{i, \eta}\|_2^2] + 4M^2\bbE[\|\beta^*\|_2^2] + 4 \trace(\varphi_i\varphi_i^\top)].
\end{align*}
By Lemma D.2 in \citet{chen2020accelerating}, if $0 < \eta < \alpha / M^2$, there exists a constant $\widetilde \delta_1(\tau_2, M, \alpha, \beta)$ such that
\[\sup_{i \geq 0}\bbE[\|\widetilde \beta_{i, \eta}\|_2^2] \leq \widetilde \delta_1(\tau_2, M, \alpha, \beta).\]
Hence,
\[\bbE[\|\nabla \widetilde G(\widetilde \beta_{i, \eta})\|_2^2] \leq \widetilde \delta_2(\tau_2, M, \alpha, \beta) + 4\trace(\varphi_i\varphi_i^\top)],\]
then
\begin{equation}\label{eq:bound-I1b1}
\int_{i\eta}^{(i+1)\eta} (s-i\eta)^2\bbE[\|\nabla \widetilde G(\widetilde \beta_{i, \eta})\|_2^2]\,ds \leq \frac{1}{3}\eta^3(\widetilde \delta_2(\tau_2, M, \alpha, \beta) + 4\trace(\varphi_i\varphi_i^\top)]).
\end{equation}
Because $\widetilde \Sigma_{i, \eta}$ is a diagonal matrix with diagonal elements $\sqrt{2\tau_1}$ or $\sqrt{2\tau_2}$, 
\[\bbE\left[\left\Vert\widetilde \Sigma_{i, \eta}\int_{i\eta}^s\,d W_t\right\Vert_2^2\right] \leq 8 \tau_2 p (s - i\eta).\]
Hence,
\begin{equation}\label{eq:bound-I1b2}
\int_{i\eta}^{(i+1)\eta} \bbE\left[\left\Vert\widetilde \Sigma_{i, \eta}\int_{i\eta}^s\,d W_t\right\Vert_2^2\right]\,ds \leq 4\tau_2 p \eta^2.
\end{equation}
Plug \eqref{eq:bound-I1b1} and \eqref{eq:bound-I1b2} into \eqref{eq:decom-I1b},
\begin{align}\label{eq:bound-I1b}
\begin{split}
\bbE[\cI_{1b}] &\leq 2M^2 \sum_{i = 1}^{\kappa} \left(\frac{1}{3}\eta^3\widetilde \delta_2(\tau_2, M, \alpha, \beta) + \frac{4p}{3}\eta^3\trace(\varphi_i\varphi_i^\top)] +  4 \tau_2 p \eta^2\right)\\
& \leq \widetilde \delta_3 \kappa p \eta^2 (1 + \eta \max_i \trace(\varphi_i\varphi_i^\top)]).
\end{split}
\end{align}
By the definition of $\varphi_k$, we have
\begin{equation}\label{eq:bound-I1c}
\cI_{1c} \leq \sum_{i = 0}^{\kappa} 2 p \eta \trace(\varphi_i\varphi_i^\top) \leq 2p\kappa\eta \max_i\trace(\varphi_i\varphi_i^\top).    
\end{equation}
Plug \eqref{eq:bound-I1a}, \eqref{eq:bound-I1b} and \eqref{eq:bound-I1c} into \eqref{eq:decom-I1},
\begin{equation}\label{eq:bound-I1}
\bbE[\cI_1] \leq 3 t M^2 \int_0^t \bbE[\|\beta_s - \widetilde \beta_{\eta}^s\|_2^2]\,ds + \widetilde \delta_4 p t(t+1) \left(\eta + \max_i\trace(\varphi_i\varphi_i^\top)]\right).
\end{equation}

For the term $\cI_2$, according to It\^{o} isometry, we have
\begin{align*}
&\bbE[\cI_2] \\
= &\sum_{j = 1}^{2p} \int_0^t \bbE[(\Sigma_s(j) - \widetilde \Sigma_{\lfloor s / \eta \rfloor, \eta}(j))^2]\,ds\\
\leq &\sum_{j = 1}^{2p} \sum_{i = 1}^{\kappa} \int_{i\eta}^{(i + 1)\eta}\bbE[(\Sigma_s(j) - \widetilde \Sigma_{i, \eta}(j))^2]\,ds\\
\leq &2 \sum_{j = 1}^{2p} \sum_{i = 1}^{\kappa}\left\{\int_{i\eta}^{(i + 1)\eta}\bbE[(\Sigma_s(j) - \Sigma_{i, \eta}(j))^2]\,ds + \int_{i\eta}^{(i + 1)\eta}\bbE[(\Sigma_{i, \eta}(j) - \widetilde \Sigma_{i, \eta}(j))^2]\,ds\right\},
\end{align*}
where we denote the $j$-diagonal entry of diagonal matrix $M$ as $M(j)$. Following the discussion of $A_4$ and $A_5$ defined in \eqref{eq:decom-A} in the proof of Theorem \ref{thm:convex-results}, we can bound $\cI_2$ by
\begin{equation}\label{eq:bound-I2}
\bbE[\cI_2] \leq \widetilde \delta_5(\tau_1, \tau_2, a) p (t+1) (\eta + \max_i \bbE[|\psi_i|]),
\end{equation}
where $\widetilde \delta_5(\tau_1, \tau_2, a) $ is a constant depends on $\tau_1$, $\tau_2$ and $a$.

Plug \eqref{eq:bound-I1}, \eqref{eq:bound-I2} into \eqref{eq:decom-I},
\begin{align*}
&\bbE[\|\beta_t - \widetilde \beta_t^{\eta}\|_2^2] \\
\leq &6 t M^2 \int_0^t \bbE[\|\beta_s - \widetilde \beta_{\eta}^s\|_2^2]\,ds + 2\widetilde \delta_4 p t(t+1) (\eta + \max_i \trace(\varphi_i\varphi_i^\top))\\
& +  \widetilde \delta_5(\tau_1, \tau_2, a) p (t + 1) (\eta + \max_i \bbE[|\psi_i|])\\
\leq &6 t M^2 \int_0^t \bbE[\|\beta_s - \widetilde \beta_{\eta}^s\|_2^2]\,ds \\
& + \widetilde \delta_6 p (t + 1) \left((t+1)\eta + t \max_i\trace(\varphi_i\varphi_i^\top) + \max_i \bbE[|\psi_i|].\right)
\end{align*}
By applying the integral form Gr\"{o}nwall’s inequality, we get
\begin{align*}
\bbE[\|\beta_t - \widetilde \beta_t^{\eta}\|_2^2] \leq \widetilde \delta_7(\tau_1, \tau_2, a, M, t, p)(\eta + \max_i\trace(\varphi_i\varphi_i^\top) + \max_i \bbE[|\psi_i|]).
\end{align*}
Thus,
\begin{align}\label{eq:discre-error}
\begin{split}
W_2(\mu_k, \nu_k) &\leq \bbE[\|\beta_{k, \eta} - \widetilde \beta_{k, \eta}\|_2] \\
& \leq \widetilde C(\tau_1, \tau_2, a, M, k, p)\left(\eta^{1/2} + \max_i\trace(\varphi_i\varphi_i^\top)^{1/2} + \max_i\bbE[|\psi_i|]^{1/2}\right),      
\end{split} 
\end{align}
together with \eqref{eq:bound-nu}, we finish the proof.
\end{proof}

\end{document}